\documentclass[12pt, reqno]{amsart}

\usepackage{amsmath, amsthm, amssymb, stmaryrd}
\usepackage{hyperref}
\usepackage{enumerate}
\usepackage{url}

\usepackage{xcolor}

\usepackage{fancyvrb}

\newtheorem{thm}{Theorem}[section]
\newtheorem{lemma}[thm]{Lemma}

\newtheorem{cor}[thm]{Corollary}
\newtheorem{conj}[thm]{Conjecture}

\theoremstyle{definition}

\newtheorem{example}[thm]{Example}

\theoremstyle{remark}
\newtheorem{remark}[thm]{Remark}

\numberwithin{equation}{section}

\def\alp{{\alpha}}

\def\del{{\delta}} \def\Del{{\Delta}}

\def\tet{{\theta}}  
\def\kap{{\kappa}}
\def\lam{{\lambda}}

\def\sig{{\sigma}}

\def\ome{{\omega}}  
\def\eps{\varepsilon}

\def\le{\leqslant} \def\ge{\geqslant}

\def \sig{{\sigma}}
\def \ind {{\mathrm{ind}}}

\def \bC {\mathbb C}

\def \bF {\mathbb F}

\def \bN {\mathbb N}

\def \bQ {\mathbb Q}
\def \bR {\mathbb R}
\def \bZ {\mathbb Z}

\def \be {\mathbf e}
	
\def \bg {\mathbf g}

\def \bt {\mathbf t}

\def \bx {\mathbf x}
\def \bw {\mathbf w}

\def \fg {\mathfrak g}

\def \cC {\mathcal C}

\def \cF {\mathcal F}

\def \cO {\mathcal O}

\def \cY {\mathcal Y}
\def \cZ {\mathcal Z}

\def \deg {\mathrm{deg}}

\def \Gal {{\mathrm{Gal}}}

\def \Fix {{\mathrm{Fix}}}

\begin{document}
\title[Enumerative Galois theory for number fields]{Enumerative Galois theory for number fields}
\author{Sam Chow \and Rainer Dietmann}
\address{Mathematics Institute, Zeeman Building, University of Warwick, Coventry CV4 7AL, United Kingdom}
\email{Sam.Chow@warwick.ac.uk}
\address{Department of Mathematics, Royal Holloway, University of London\\
Egham TW20 0EX, United Kingdom}
\email{Rainer.Dietmann@rhul.ac.uk}

\begin{abstract} 
Counting number fields with prescribed Galois group is an enduring challenge in arithmetic statistics. Using the determinant method, we provide an upper bound for even groups, which is new in some cases.
\end{abstract}

\maketitle

{\centering\footnotesize \em{Dedicated to Trevor Wooley on the occasion of his diamond jubilee, with gratitude for all he has done for us and for mathematics} \par}

\section{Introduction}

Fix an integer $n \ge 2$ and, throughout, let $X \in \bN$ be large. Let $F_n(X)$ count number fields $K$ such that 
\[
[K:\bQ] = n, \qquad |\Del_K| \le X,
\]
where $\Del_K$ is the discriminant of $K$ and our numbers fields all lie in a fixed algebraic closure of the rationals. By the Hermite--Minkowski theorem, the quantity $F_n(X)$ is finite. 

An old folklore conjecture asserts that $F_n(X) \sim c_n X$, for some $c_n > 0$. This is known for $n \le 5$, through the works of Davenport--Heilbronn~\cite{DH1971}, Cohen--Diaz y Diaz--Olivier \cite{CDO2002}, and Bhargava \cite{Bha2005, Bha2010}. For larger values of $n$, the problem remains open. 

Schmidt~\cite{Sch1995b} famously showed that
\[
F_n(X) \ll_n X^{\sig_n},
\]
where
\[
\sig_n = \frac{n+2}{4}.
\]
This was subsequently improved by Ellenberg and Venkatesh~\cite{EV2006}, Couveignes~\cite{Cou2020}, and most recently by Lemke Oliver and Thorne~\cite{LOT2022}. Lemke Oliver and Thorne's bound $O_n(X^{1.564(\log n)^2})$ is best known for $n \ge 159$. We state these bounds in Table~\ref{SchmidtTable}, so as to list the best bounds for large, fixed $n$. Throughout, the parameter $C = C_n \in \bN$ denotes a large constant.
\begin{table}[ht]
\begin{tabular}{|c|c|c|}
\hline
Year & Authors & $F_n(X) \ll_n$ \\
\hline
$1995$ & Schmidt & $X^{\sig_n}$ \\
$2006$ & Ellenberg, Venkatesh & $\exp(C \sqrt{\log n})$ \\
$2019$ & Couveignes & $X^{C (\log n)^3}$ \\
$2020$ & Lemke Oliver, Thorne & $ X^{1.564(\log n)^2}$ \\ 
\hline
\end{tabular}
\caption{\label{SchmidtTable}Best upper bounds for $F_n(X)$ when $n$ is large}
\end{table}

In the intermediate range $6 \le n \le 158$, the best known bound 
\[
F_n(X) \ll_{n,\eps} X^{\ome_n + \eps}
\]
was recently demonstrated by Bhargava, Shankar and Wang \cite{BSW2022}, where
\[
\ome_n = \sig_n - \: \frac1{2n-2} + \frac1{2^{2\lfloor(n-1)/2\rfloor}(2n-2)}.
\]
Concurrently, a weaker bound of this type was obtained in \cite{Kevin}.

The estimates $F_n(X) \sim c_n X$ for $n = 3, 4, 5$ were established by stratifying the count according to the Galois group of the polynomial. The latter is an interesting problem in its own right, as we discuss next. Our results are new bounds for certain groups.

\subsection{A Galois counting problem for number fields}

For $G$ a transitive subgroup of $S_n$, write $\cF_n(X;G)$ for the set of number fields $K$ such that
\begin{equation} \label{GExt}
[K:\bQ] = n,
\qquad
\Gal(\hat K / \bQ) = G, \qquad |\Del_K| \le X,
\end{equation}
and write $F_n(X;G)$ for the cardinality of this set. Here we fix an algebraic closure $\overline \bQ$, and $\hat K \le \overline \bQ$ denotes the Galois closure of $K$ over $\bQ$. The Galois group $\Gal(\hat K / \bQ)$ is considered as a permutation group on the $n$ embeddings of $K$ into $\overline \bQ$, and equality is up to conjugacy.

Malle's conjecture \cite{Mal2002, Mal2004} predicts the asymptotic behaviour of the quantity $F_n(X;G)$. To begin with, we state a weak form of it below, as given in \cite{Alb2020}. Define
\[
\ind(G) = \min \{ \ind (g): 1 \ne g \in G \},
\]
where $n - \ind(g)$ denotes the number of orbits of $g$ on $\{1,2,\ldots,n\}$.

\begin{conj} [Weak Malle conjecture over the rationals]
\label{WeakMalle}
For any $\eps > 0$,  we have
\[
X^{1/\ind(G)} \ll_n F_n(X;G) \ll_{n,\eps} X^{\eps + 1/\ind(G)}.
\]
\end{conj}

Malle's prediction was in fact stronger.

\begin{conj} [Malle's conjecture over the rationals]
\label{Malle}
There exist $a,b,c \in \bZ$, depending on $G$, such that
\[
F_n(X;G) \sim cX^{1/a} (\log X)^{b-1} \qquad (X \to \infty).
\]
\end{conj}

\noindent
This formulation is widely believed to this day. Malle's predicted value $a = \ind(G)$ is also widely believed, but there has been some dispute over the values of $b$ and $c$. Famously, Kl\"uners's \cite{Klu2005} gave a counterexample to Malle's initial suggested value of $b$. T\"urkelli \cite{Tur2015}, based on a function field analysis, proposed another value of $b$. In our setting, this has since been contradicted and refined by Wang \cite{Wan}. (However, over function fields, T\"urkelli's value of $b$ has been validated by Landesman and Levy~\cite{LL}.) The value of $c$ has seemed mysterious up until recently, when Loughran and Santens \cite{LS} used Brauer groups of stacks to put forward a conjectural value. This comes from an interpretation in terms of Manin's conjecture, continuing the theme of \cite{ESZ2023}.

Malle's conjecture is known to hold for abelian groups \cite{Mak1985, Wri1989} and certain low-degree groups \cite{Bha2005, Bha2010, DH1971}.
Some other cases are also known \cite{ALWW, FK2021, KP2023, MTTW, Wan2021}.
Nonetheless, the problem remains very much open. 

Let us now turn to groups for which Malle's conjecture is not known. Mehta obtained bounds in some special cases \cite{Meh2019}. For certain groups, Dummit \cite{Dum2018} established a bound on $F_n(X;G)$ that is smaller than Schmidt's bound on $F_n(X)$. 

\begin{thm} [Dummit] \label{DummitThm} Let $G$ be a proper, transitive subgroup of $S_n$. Let $t$ be such that if $K \in \cF_n(X;G)$ and $K'$ is a proper subfield of $K$ then $[K':\bQ] \le t$. Then
\begin{equation*}
F_n(X;G) \ll_{n,\eps} X^{D_G + \eps},
\end{equation*}
where
\[
D_G = \inf \left\{
\begin{aligned}
&\frac{1}{2(n - t)} \left( -1 + \sum_{i=2}^{n} \deg(f_i) \right) : \\
&f_1, \ldots, f_n\ \text{is a set of primary invariants for } G
\end{aligned}
\right\}.
\]
Further, one can take
\[
\deg(f_i) \le i \qquad (1 \le i \le n).
\]
\end{thm}

\noindent See \cite[\S 3]{Dum2018} for a definition and discussion of primary invariants.

In \S \ref{PrimitiveGroups}, we deduce the following consequence of Dummit's theorem.

\begin{cor} \label{DummitCor} Let $G$ be a proper, transitive subgroup of $S_n$. Then
\[
F_n(X;G) \ll_{n,\eps} X^{\del_n +\eps},
\]
where
\[
\del_n = \sig_n - \: \frac1{2n-2}.
\]
\end{cor}

Observe that $\del_n < \ome_n$, so Dummit's bound for $F_n(X;G)$ is also smaller than the Bhargava--Shankar--Wang bound for $F_n(X)$. Recently, Bhargava \cite[Theorem 20]{Bha} produced the following improvement. To state it, we require the notion of a primitive group \cite{Cox, DM1996}.
\begin{enumerate}[(a)]
\item A transitive group $G \le S_n$ is \emph{imprimitive} if $\{1,2,\ldots,n\}$ is the disjoint union of non-empty sets $R_1,\ldots,R_k$, where $k \ge 2$ and $\max |R_j| \ge 2$, such that if $\tau \in G$ and \mbox{$i \in \{1,2,\ldots,k\}$} then $\tau(R_i) = R_j$ for some $j \in \{1,2,\ldots,k\}$.
\item A transitive group $G \le S_n$ is \emph{primitive} if it is not imprimitive.
\end{enumerate}
Note that if $G$ is imprimitive then the blocks $R_1,\ldots,R_k$ all have the same size.

\begin{example}
The transitive subgroups of $S_4$ are $S_4, A_4, D_4, V_4$, and $C_4$, where $D_4$ denotes the dihedral group of order 8 and $V_4$ denotes the Klein four-group. Observe that 
\[
D_4 = \langle (1234), (13) \rangle
\]
preserves the partition
\[
R_1 = \{ 1,3 \}, \qquad
R_2 = \{ 2, 4 \}
\]
and is therefore imprimitive. Since $V_4$ and $C_4$ are subgroups of $D_4$, they too are imprimitive. On the other hand, the group $A_4$ contains all three-cycles on four elements and therefore does not preserve any non-trivial partition. Therefore $A_4$ is primitive and so too is $S_4$.
\end{example}

\begin{thm} [Bhargava]
Let G be a primitive permutation group on n letters. Then
\[
F_n(X;G) \ll_{n,\eps} X^{B + \eps},
\]
where
\[
B = B(G) = \sig_n - 1 + \ind(G)^{-1}.
\]
\end{thm}

Though Bhargava restricts attention to the case of primitive permutation groups, no generality is lost in doing so, as we now explain. A number field is \emph{primitive} if there are no fields strictly between it and the rationals. As noted in \cite{BSW2022}, imprimitive fields can be counted via their proper primitive subextensions \cite[Equation (1.2)]{Sch1995b}, giving rise to the bound
\[
O_n(X^{(n/2+2)/4}) = O_n(X^{(n+4)/8}).
\]
Moreover, it is known that primitive fields have primitive Galois groups, and for completeness we provide a proof of this in \S \ref{PrimitiveGroups}. Hence, we reach the following conclusion.

\begin{lemma} \label{imprim}
If $G \le S_n$ is imprimitive then
\[
F_n(X;G) \ll_n X^{(n+4)/8}.
\]
\end{lemma}

For $d \in \bN$, define
\[
B_n(d) = \sqrt{\frac{n}{4(n-1)d}},
\qquad
N_n = \frac{n}{4n-4}.
\]
Recall that $F_n(X) \ll_n X$ for $n \le 5$. We establish the following bound.

\begin{thm} \label{MainThm} Assume that $n \ge 6$, let $G$ be a proper, transitive subgroup of $S_n$, and put $d = [S_n:G]$. Assume that $G < A_n$ and
\[
n \notin 
\bigcup_{m \text{ odd}} 
\{ m^2, m^2 + 1 \}.
\]
Then
\[
F_n(X;G) \ll_{n,\eps} X^{E + \eps},
\]
where
\[
E = E(G) = \sig_n - 1 - \:  \frac1{2n-2} + B_n(d) + N_n.
\]
\end{thm}

This refines Bhargava's bound in three cases, called 6T12, 7T5 and 8T48, see \cite{BM1983}. These are precisely the primitive groups $G$ of degree $n \ge 6$ such that $\ind(G) = 2$ and $G < A_n$, see \cite{BM1983}.

\begin{table}[ht]
\begin{tabular}{|c|c|c|c|c|}
\hline
$G$ & $B(G)$ & $E(G)$ & $|G|$ \\
\hline
6T12 & 1.5 & 1.36 & 60 \\
7T5 & 1.75 & 1.56 & 168 \\
8T48 & 2 & 1.81 & 1344 \\ 
\hline
\end{tabular}
\caption{\label{RecordTable}
Exponents to two decimal places.}
\end{table}
\noindent 

\begin{remark} After the initial release of this manuscript, the upper bound $O(X)$ for 6T12 was asserted in \cite{ALWW}, citing a quintic result \cite{BSW2015}. Having checked the GitHub repository attached to \cite{ALWW}, we believe that our bounds for 7T5 and 8T48 remain the best known.
\end{remark}

\begin{remark}
Our methods are geared towards slight improvements over Schmidt's bound, when the Galois group is fixed. For this reason, there are only a few cases in which we obtain state-of-the-art bounds. In the initial release of this manuscript \cite{arXiv}, we used a different version of the machinery to also obtain a state-of-the-art bound for 6T14. However, we later decided to omit this case, as it was lengthy, computed-assisted, and had been obsoleted by \cite{ALWW}. It is plausible that improvements in the determinant method for threefolds --- such as a proof of the statement analogous to Theorem \ref{BrowningThm} ---
could lead to further developments. 
\end{remark}

In view of Conjecture \ref{WeakMalle}, our bounds are likely far from the truth, since $\ind(G) = 2$ for the groups above. Lower bounds are also interesting \cite{Alb2021, BSW2015, BSW2022, LLOT, PTBW}. All of these counting problems have a rich history, and are intricately related to many other problems in arithmetic statistics \cite{Pie}.

Finally, we note with great interest the programme laid out by Lemke Oliver \cite{LO}, to attack the general number field counting problem. After stratifying by Galois group, an appeal is made to the classification of finite simple groups. Impressive bounds are then obtained in many general families. For example, it is shown that
\[
F_n(X; G) \ll_n X^{14}
\]
if $G$ is solvable.

\subsection{Methods}

Schmidt's approach \cite{Sch1995b} allows us to count primitive number fields with Galois group $G$ by counting polynomials $f$ with Galois group $G$. The coefficients $a_i$ of these polynomials $f$ are integers in lopsided boxes. We then introduce resolvents $\Phi$, which are polynomials in the coefficients $a_i$, and have the property that if $f$ has Galois group $G$ then $\Phi$ has an integer root that is polynomially bounded in terms of the height of $f$. This reduces the problem to counting solutions to a diophantine equation in the $a_i$ and an extra variable $y$.

The idea is to fix all but two of the $a_i$, leaving us to count integer points on a surface cut out by a polynomial which may have large coefficients, up to a given height. We can cover almost all such points by curves using the determinant method, with the number of curves required being uniformly controlled. If our fixed $a_i$ are sufficiently generic, then our curves are non-linear. Finally, we can apply a lopsided version of a celebrated estimate of Bombieri and Pila to count integer points on such a curve. 

These ingredients are presented in \S \ref{ingredients}. Here, the phrase `sufficiently generic' entails in particular that the surface is absolutely irreducible and contains no rational lines. Handling these two aspects is challenging in general. Since our polynomials have even Galois group, we may exploit the fact that the discriminant is a square, to refine the basic strategy described above.

Choosing all but two of the $a_i$ produces a polynomial with Galois group $S_n$ over a two-parameter function field. We do so in two steps, first choosing all but three of the $a_i$, then choosing one more generically. This ensures that a certain resolvent is separable, and we are able to deduce from this the absolute irreducibility of the surface obtained. 

We then apply the determinant method to cover almost all of its integer points up to a given height by curves. If the curve is non-linear, then we can proceed as in the general framework through a lopsided Bombieri--Pila estimate. If the curve is linear, then we can substitute the linear relation into the equation of the discriminant being a square. The polynomial obtained is irreducible, by some previous work of the second-named author \cite{Die2013}, and lopsided Bombieri--Pila prevails.

\subsection{Notation}
We adopt the convention that $\eps$ is an arbitrarily small positive constant whose value may change between occurrences. Throughout $X$ denotes a positive real number, sufficiently large in terms of $n$ and $\eps$. We use the Vinogradov and Bachmann--Landau asymptotic notations.

\subsection{Funding and acknowledgements} 
SC was supported by EPSRC Fellowship Grant EP/S00226X/2 and thanks Samir Siksek for helpful discussions. Both authors are indebted to the Mathematisches Forschungsinstitut Oberwolfach for favourable working conditions, and thank the anonymous referees for constructive comments.

\section{Ingredients}
\label{ingredients}

\subsection{Primitive fields and groups} \label{PrimitiveGroups}

In this subsection, we establish an equivalence between primitive fields and groups, as well as deducing Corollary \ref{DummitCor}.

\begin{lemma} \label{PrimitiveLemma} Let $K$ be a number of field of degree $n$, let $\hat K$ be its Galois closure, and suppose $\Gal(\hat K/\bQ) = G \le S_n$. Then $G$ is primitive if and only if $K$ is primitive.
\end{lemma}

\begin{proof} By the primitive element theorem, we have $K = \bQ(\alp)$ for some $\alp = \tet_{1,1} \in K$. Let $\iota_1,\ldots,\iota_n$ be the embeddings of $K$ into $\hat K$. 

($\Leftarrow$) First suppose $G$ is imprimitive. Let $f$ be the minimal polynomial of $\alp$. The conjugates of $\alp$ are the other roots of $f$, namely $\iota_i \alp$ ($1 \le i \le n$). Hence $G$ acts transitively and imprimitively on these roots. Thus, for some integer $k \in [2,n-1]$ dividing $n$, and for $m = n/k$, we have an induced action of $G$ on $\{R_1,\ldots,R_k\}$, where
\[
R_i = \{ \tet_{i,1}, \ldots, \tet_{i,m} \} \qquad (1 \le i \le k).
\]
Define $g(x) = \prod_{\tet \in R_1} 
(x-\tet)$, and let $N$ be the field generated by the coefficients of $g$ over $\bQ$. Then $N$ contains $\bQ$. This containment is strict, as $f$ is irreducible, and moreover
\[
N \le \hat K^{\Fix(R_1)} \le \hat K^{\Fix(\alp)} = \hat K^{\Fix(K)}
= K.
\]
Assume for a contradiction that $N = K$. Then $\alp \in N$, so $\alp$ is a symmetric polynomial in the elements of $R_1$. Consequently, if \mbox{$\sig_1,\sig_2 \in G$} and $\sig_1(R_1) = \sig_2(R_1)$ then $\sig_1(\alp) = \sig_2(\alp)$, whence 
$|G \alp| \le k$. On the other hand, as $G$ acts transitively on the roots of $f$, we have 
$|G \alp| = n > k$, contradiction. Hence $N$ lies strictly between $K$ and $\bQ$, and so $K$ is imprimitive.

($\Rightarrow$) Now suppose instead that $K$ is imprimitive. Then there exists a field $N$ strictly between $K = \bQ(\alp)$ and $\bQ$. Put
\[
H = \Gal(\hat K/N),
\qquad
\Del = H \alp.
\]
Note from the Galois correspondence that 
\[
\Fix(\alp) \le H,
\]
and let $\sig_1, \ldots, \sig_m$ be left coset representatives for $\Fix(\alp)$ in $H$. Then
\[
\Del = \{ \sig_1 \alp, \ldots, \sig_m \alp \},
\]
so
\[
|\Del| = m = [H: \Fix(\alp)] = [K:N] \in [2,n-1].
\]

We claim that $\Del$ is a block for the action of $G$, meaning that if $g \in G$ then 
\[
g \Del = \Del
\qquad
 \text{or}
\qquad
g \Del \cap \Del = \emptyset.
\]
To see this, suppose $g\Del$ intersects $\Del$. Then $gh_1 \alp = h_2 \alp$ for some $h_1, h_2 \in H$. Then
\[
h_2^{-1} g h_1 \in \Fix(\alp) \le H,
\]
and if $h \in H$ then 
\[
g (h \alp) = 
h_2 (h_2^{-1} g h_1) h_1^{-1} h \alp \in H \alp = \Del.
\]
Hence $g \Del \subseteq \Del$, and so $g \Del = \Del$, confirming the claim.

Consequently, if $g_1, g_2 \in G$ then
\[
g_1 \Del = g_2 \Del
\qquad
\text{or}
\qquad g_1 \Del \cap g_2 \Del = \emptyset.
\]
Thus, the sets $g \Del$ for $g \in G$ partition the conjugates of $\alp$ into
$k = n/m$ disjoint sets
\[
R_1, \ldots, R_k,
\]
each of size $m$. Finally, observe that if 
$i \in \{1,2,\ldots,k\}$ and $\tau \in G$ then 
$\tau(R_i) = R_j$ for some $j \in \{1,2,\ldots,k\}$, and so $G$ is imprimitive.
\end{proof}

\begin{proof} [Proof of Corollary \ref{DummitCor}]
By \cite[Equation (1.2)]{Sch1995b} and Lemma~\ref{PrimitiveLemma}, if $G$ is imprimitive then
\[
F_n(X;G) \ll_n X^{(n/2+2)/4}.
\]
If $G$ is primitive then, by Lemma \ref{PrimitiveLemma}, we have $t=1$ in Theorem \ref{DummitThm}, whence
\[
F_n(X;G) \ll_{n,\eps} X^{D+\eps},
\]
where
\[
D = \frac1{2n-2} \left( \frac{n(n+1)}2 \: - 2 \right)
= \frac{n+2}{4} \: - \: \frac1{2n-2} = \del_n.
\]
\end{proof}

\subsection{From number fields to polynomials}

Given an irreducible polynomial $f(x) \in \bZ[x]$, we denote its Galois group over $\bQ$ by $G_f$. This is the Galois group of its splitting field. Thus, if $f$ has degree $n$, then $G_f$ is a transitive subgroup of $S_n$ via its action on the roots of $f$. Recently, there have been several articles counting polynomials with prescribed Galois group~\cite{And2021, BarySoroker, CD2020, CD, Xiao}, including the resolution of a weak form of van der Waerden's conjecture by Bhargava \cite{Bha}. The polynomial counting problem that arises in the present work differs in that the coefficients are drawn from lopsided boxes.

Schmidt \cite{Sch1995b} uses the geometry of numbers to pass from counting number fields to counting polynomials. We extract a traceless version of this key ingredient from \cite[\S 2]{LOT2022}.

\begin{lemma} \label{traceless} Let $K$ be a primitive number field of degree $n$, and let $\cO_K$ be its ring of integers. Then there exists $\alp \in \cO_K$ whose minimal polynomial is
\begin{equation} \label{form}
x^n + a_2 x^{n-2} + a_3 x^{n-3} + \cdots + a_n \in \bZ[x],
\end{equation}
where
\begin{equation*}
a_j \ll |\Del_K|^{j/(2n-2)}
\qquad (2 \le j \le n).
\end{equation*}
\end{lemma}

This enables the passage to counting polynomials as follows.

\begin{cor} \label{StartingPoint} Let $G$ be a primitive subgroup of $S_n$. Then $F_n(X;G)$ is bounded above by the number of irreducible polynomials
\eqref{form} with Galois group $G$ satisfying 
\begin{equation} \label{ranges}
|a_j| \le CX^{j/(2n-2)}
\qquad (2 \le j \le n).
\end{equation}
\end{cor}

\begin{proof} Lemma \ref{traceless} defines a function from $\cF_n(X;G)$ to the set of irreducible polynomials \eqref{form} with coefficients satisfying \eqref{ranges}. The image of $K$ is the minimal polynomial of some $\alp \in K \setminus \bQ$, where $K = \bQ(\alp)$ by Lemma \ref{PrimitiveLemma}. The polynomial must have splitting field $\hat K$ and therefore Galois group $G$. The function bijects onto its image, since we have $K = \bQ(\alp)$ whenever $K$ is sent to the minimal polynomial of $\alp$. We have shown that $\cF_n(X;G)$ injects into the set of irreducible polynomials \eqref{form} with Galois group $G$ and coefficients in the ranges \eqref{ranges}.
\end{proof}

Corollary \ref{StartingPoint} directly implies Schmidt's bound $F_n(X;G) \ll X^{\sig_n}$, and serves as the starting point for the proof of Theorem \ref{MainThm}.

\subsection{Galois theory over function fields}

There is a close connection between the Galois group of a polynomial over a function field, say $f(\bt, x) \in \bF(\bt)[x]$, for some field $\bF$, and its Galois group when the variables $\bt$ are specialised. Specifically, the Galois group of a specialisation should generically be the same as the Galois group in the function field $\bF(\bt)$. Results of this kind date back to Hilbert's irreducibility theorem. When it comes to quantitative bounds on the number of exceptional specialisations, considerable progress was made by Cohen in the late 1970s and 1980s, see in particular \cite{Coh1981}.
We require two preliminary results from this topic. 

The following is a special case of a result over two-parameter function fields \cite[Theorem 1]{Coh1980}.

\begin{thm}
\label{TwoParameter}
Let $\bF$ be a field of characteristic zero, and let
\[
g(x) = x^n + a_1 x^{n-1} + \cdots + a_{n-2} x^2 \in \bF[x].
\]
Then the Galois group of
\[
g(x) + t x + u
\]
over $\bF(t,u)$ is $S_n$.
\end{thm}

We also use one-parameter function fields, specifically \cite[Lemma 2]{Die2012} as stated below.

\begin{lemma} \label{OneParameter}
Let $n, r \in \bN$ be coprime with $n > r$. Let 
\[
a_j \in \bZ \qquad (1 \le j \le n-1, \quad j \ne r).
\]
Then, for all but $O_n(1)$ integers $a_r$, the polynomial
\[
x^n + a_1 x^{n-1} + \cdots + a_{n-1} x + t
\]
has Galois group $S_n$ over $\bQ(t)$.
\end{lemma}

\subsection{Resolvents}
\label{ResolventBit}

A resolvent is an auxiliary polynomial whose factor type informs the Galois group of the original polynomial. For a classical introduction, see \cite{Cox}.

If $x^n + a_1 x^{n-1} + \cdots + a_n \in \bZ[x]$ has Galois group $G$ and roots $\alp_1, \ldots, \alp_n \in \bC$, then the polynomial
\begin{equation} \label{PhiDef}
\Phi(y) = \Phi(y; a_1, \ldots, a_n) = \prod_{\sig \in S_n / G}
\left(y - \sum_{\tau \in G} \prod_{i \le n} \alp_{\sig \tau(i)}^i \right)
\in \bZ[y;a_1,\ldots,a_n]
\end{equation}
has an integer root $y$. This was used in \cite{Die2012}. By \cite[Lemma 4.1]{CD}, the total degree of $\Phi$ is $O_n(1)$, and if we also have $a_1,\ldots,a_n \in [-X,X]$ then $\Phi$ has an integer root $y \ll X^{O_n(1)}$.

The construction was generalised in \cite{CD2017}. For
\[
\bw = (w_1, \ldots, w_{|G|}) \in \bN^{|G|},
\qquad
\be = (e_1,\ldots,e_n) \in \bN^n,
\qquad \fg \in \bZ,
\]
define
\begin{align*}
\Phi_{\bw, \be, \fg}(y) &= 
\Phi_{\bw, \be, \fg}(y;
a_1, \ldots, a_n) \\ &=
\prod_{\sig \in S_n/G}
(y - r_{\bw, \be, \bg}(\sig))
\in \bZ[y; a_1,\ldots,a_n],
\end{align*}
where
\[
r_{\bw,\be,\fg}(\sig)
= \sum_{k \le |G|} w_k
\sum_{\tau \in G}
\prod_{i \le n}
(\alp_{\sig \tau(i)} + \fg)^{ke_i}.
\]
The present authors studied these resolvents further in \cite{CD}. We now state \cite[Lemma 4.3]{CD}. We also include an explicit description of the roots of the resolvent, which comes from the proof.

\begin{lemma}  \label{GeneralResolvent}
Let
\[
g(T_1,\ldots,T_s,x)
\in \bZ[T_1,\ldots,T_s,x]
\]
be separable and monic of degree $n \ge 1$ in the variable $x$, let $G$ be its Galois group over $\bQ(T_1,\ldots,T_s)$, and let $K \le G$. Let $D$ be the total degree of $g$. Then there exists
\[
\Phi_{g,K}(T_1,\ldots,T_s,Y) \in \bZ[T_1,\ldots,T_s,Y]
\]
of total degree $O_{s,D}(1)$,
monic of degree $[S_n:K]$ in $Y$, such that:
\begin{enumerate}[(i)]
\item Each irreducible divisor has degree at least $[G:K]$ in $Y$.
\item The roots $y \in \overline{\bQ(T_1,\ldots,T_s)}$ are given by
\begin{equation} \label{roots}
r_\sig = \sum_{k \le |K|} w_k
\sum_{\tau \in K} 
\prod_{i \le n}
(\alp_{\sig \tau(i)} + \fg)^{k e_i}
\qquad (\sig \in S_n/K)
\end{equation}
for some positive integers $w_1, \ldots, w_{|K|}, e_1, \ldots, e_n \ll_D 1$ and $\fg \ll_D |g|$, where 
\[
\alp_j = \alp_j(T_1,\ldots,T_s)
\in \overline{\bQ(T_1,\ldots,T_s)}
\qquad (1 \le j \le n)
\]
are the roots of $g(x) \in
\bZ[T_1,\ldots,T_s][x]$.
\item If $t_1,\ldots,t_s \in \bZ$ and $g(t_1,\ldots,t_s,X)$ has Galois group $K$ over $\bQ$ then
$\Phi_{g,K}(t_1,\ldots,t_s,Y)$ has an integer root 
\begin{equation} \label{ybound}
y \ll_D ( |g| \cdot \| \bt \|_{\infty})^{O_{D}(1)}.
\end{equation}
\end{enumerate}
\end{lemma}

\subsection{The determinant method}
\label{DeterminantMethodBit}

This is a tool to count solutions to diophantine equations. It often delivers upper bounds that are uniform in the sizes of the coefficients. For an overview, see \cite{Cetraro}. As discussed in the introduction, the determinant method will play a leading role in our arguments.

The archetypal application of the determinant method is the following theorem \cite{BP1989}.

\begin{thm} [Bombieri--Pila 1989] \label{BPthm} Let $\cC$ be an absolutely irreducible algebraic curve of degree $d \ge 2$ in $\bR^2$. Then the number of integer points in $\cC \cap [-H,H]^2$ is $O_{d,\eps}(H^{\eps+1/d})$.
\end{thm}

\noindent We use this and its lopsided generalisation, as stated in \cite[Lemma 8]{Die2012}. This was essentially given by Browning and Heath-Brown \cite{BHB2005}, however a short argument was incorporated in order to relax the absolute irreducibility requirement to irreducibility over $\bQ$.

\begin{thm} [Lopsided Bombieri--Pila] \label{LopsidedBP}
Let $F \in \bZ[x_1,x_2]$ be irreducible over $\bQ$ and have degree $d \in \bN$. Further, let $B_1, B_2 \ge 1$, and define
\[
N(F; B_1, B_2) = \: \# \{\bx \in \bZ^2:
F(\bx) = 0, \: |x_i| \le B_i \: (1 \le i \le 2) \}.
\]
Put
\[
T = \max \left \{ B_1^{e_1} B_2^{e_2} \right \},
\]
where the maximum is taken over all $(e_1,e_2) \in \bZ_{\ge 0}^2$ for which $x_1^{e_1} x_2^{e_2}$
occurs in $F(\bx)$ with non-zero coefficient. Then
\[
N(F; B_1, B_2) \ll_{d, \eps} T^\eps \exp \left( \frac{\log B_1 \cdot \log B_2}{\log T} \right).
\]
\end{thm}

Browning \cite[Lemma 1]{Bro2011} states the following result, attributing it to Salberger.

\begin{thm} \label{BrowningThm}
Suppose $F(x_1, x_2, x_3) \in \bZ[\bx]$ defines an absolutely irreducible surface of degree $d \in \bN$, and let $B_1,B_2,B_3 \ge 1$. Put
\[
T = \max \left \{ B_1^{e_1} B_2^{e_2} B_3^{e_3} \right \},
\]
where the maximum is taken over all triples $(e_1, e_2, e_3) \in \bZ_{\ge 0}^3$ for which $x_1^{e_1} x_2^{e_2} x_3^{e_3}$
occurs in $F(\bx)$ with non-zero coefficient, and also put
\[
V_3 = \exp \left \{
\left(\frac{\log B_1 \cdot \log B_2 \cdot \log B_3}{\log T} \right)^{1/2}
\right\}.
\]
Then there exist $g_1(\bx), \ldots, g_J(\bx) \in \bZ[\bx]$ and $\cZ \subset \bZ^3$, with 
\[
J \ll_{d,\eps} T^\eps V_3,
\qquad
|\cZ| \ll_{d,\eps} T^\eps V_3^2,
\]
such that the following hold:
\begin{enumerate}[(i)]
\item Each $g_j$ is coprime to $F$ and has degree $O_{d,\eps}(1)$.
\item If $\bx \in \bZ^3 \cap [-B_1,B_1] \times [-B_2,B_2] \times [-B_3, B_3] \setminus \cZ$ and $F(\bx) = 0$ then $g_j(\bx) = 0$ for some $j$.
\end{enumerate}
\end{thm}

The main strengths of the determinant method are (i) the ability to deliver bounds that are uniform in the coefficients, and (ii) versatility. The latter is also a weakness, as stronger bounds can often be obtained in more specific settings, using the structure of the polynomials involved. Another weakness of the determinant method is that optimal bounds are often not yet known, for example the analogue of Theorem~\ref{BrowningThm} for threefolds remains unsolved.

\subsection{Linear relations on the discriminant variety}

We will use the determinant method to cover the integer points up to a given height on a surface by a collection of curves. As discussed in the introduction, rational lines need to be treated separately. We will handle these cases using the following tools.

Below we state \cite[Lemmas 5 and 6]{Die2013}, incorporating the additional conditions discussed in \cite[\S 7]{DOS}. We denote by $\Del(a_1,\ldots,a_n)$ the discriminant of 
\[
x^n + a_1 x^{n-1} + \cdots + a_n.
\]

\begin{lemma} \label{even1} Let $n \ge 3$ and $a_1, \ldots, a_{n-2}$ be integers with \[
n \notin 
\bigcup_{m \text{ odd}} 
\{ m^2, m^2 + 1 \},
\]
and let $c_1, c_2 \in \bQ$. Then the polynomial
\[
z^2 - \Del(a_1, \ldots, a_{n-2}, c_1 a_n + c_2, a_n)
\]
in $z$ and $a_n$ is irreducible over $\bQ$.
\end{lemma}

\begin{lemma} \label{even2} Let $n \ge 3$ and $a_1, \ldots, a_{n-2}$ be integers such that $n$ is not an odd square plus one, and let $c \in \bQ$. Then the polynomial
\[
z^2 - \Del(a_1, \ldots, a_{n-2}, a_{n-1}, c)
\]
in $z$ and $a_{n-1}$ is irreducible over $\bQ$.
\end{lemma}

\section{Proof of Theorem \ref{MainThm}}

Observe that 
\[
4 \le d \ll_n 1.
\]
One can check that
\[
\frac{n+4}{8} \le E.
\]
This is straightforward to check if $n \ge 7$. In the case $n = 6$, the inequality follows readily upon noting from the data tabulated in \cite{BM1983} that $d \le 120$. Thus, by Lemma \ref{imprim} and Corollary \ref{StartingPoint}, it suffices to count integers $a_2,\ldots,a_n$ in the ranges \begin{equation} \label{RangesAgain}
|a_j| \le CX^{j/(2n-2)}
\qquad (2 \le j \le n).
\end{equation}
such that 
\[
x^n + a_2 x^{n-2} + a_3 x^{n-3} + \cdots + a_n \in \bZ[x]
\]
has Galois group $G$. 

We begin by choosing $a_2, \ldots, a_{n-3}$ arbitrarily in the ranges \eqref{RangesAgain}, in one of 
\[
O_n \left(X^{
\sig_n - 1 - 1/(2n-2) 
- (n-2)/(2n-2)}
\right)
\]
many ways. 
Define
\[
f_3(x; a, b, c)
= x^n + a_2 x^{n-2} + \cdots
+ a_{n-3} x^3 + ax^2 + bx + c 
\in \bZ[x; a,b,c].
\]
By Theorem \ref{TwoParameter}, applied to the field $\bQ(a)$, the polynomial $f_3$ has Galois group $S_n$ over $\bQ(a,b,c)$. Lemma \ref{GeneralResolvent} furnishes 
\[
\Phi_{f_3}(A,B,C,Y) 
\in \bZ[A,B,C,Y],
\]
monic of degree $d$ in $Y$ and irreducible over $\bZ(A,B,C)$, such that if $a,b,c \in \bZ \cap [-X,X]$ and $f_3(x; a,b,c)$ has Galois group $G$ over $\bQ$ then $\Phi_{f_3}(a,b,c,Y)$ has an integer root $y \ll_n X^{O_n(1)}$. By Gauss's lemma, it is also irreducible over the quotient field $\bQ(a,b,c)$. This field is perfect, so the discriminant of $\Phi_{f_3}$ in $y$ is a non-zero integer polynomial in $a,b,c$. In particular, all but $O_n(1)$ specialisations $a \in \bZ$ give rise to a separable polynomial $g(b,c,y) = \Phi_{f_3}(a,b,c,y)$ in $y$ over $\bQ(b,c)$.

\subsection{Separable case}

We now choose an integer
$a \ll X^{(n-2)/(2n-2)}$ for which 
$g(b,c,y)$ is separable over $\bQ(b,c)$ as a polynomial in $y$. By Theorem \ref{TwoParameter}, the polynomial 
\[
f_2(x;b,c) := f_3(x;a,b,c)
\]
has Galois group $S_n$ over $\overline{\bQ}(b,c)$. 

Assume for a contradiction that the affine surface $\cY$ cut out by the vanishing of $g(b,c,y)$ is not absolutely irreducible. Then, as $g(b,c,y)$ is monic in $y$, we must have
\[
g(b,c,y) = 
h_1(y; b,c) h_2(y; b,c) \in \overline{\bQ}(b,c)[y],
\]
for some non-constant polynomials $h_1$ and $h_2$. For $i=1,2$, let $r_{\sig_i}$ be a root of $h_i$, given by specialising $a$ in \eqref{roots}. Then, with 
\[
\kap = \sig_2 \sig_1^{-1} \in \Gal(f_2,\overline{\bQ}(b,c)),
\]
we have $\kap(r_{\sig_1}) = r_{\sig_2}$. As $h_1$ has coefficients in $\overline{\bQ}(b,c)$, we see that $h_1$, $h_2$ have a common root $\kap(r_{\sig_1}) = r_{\sig_2}$, contradicting the separability of $g$. The upshot is that $\cY$ is indeed absolutely irreducible.

By Theorem \ref{BrowningThm}, there exist $g_1(b,c,y),\ldots,g_J(b,c,y) \in \bZ[b,c,y]$ and $\cZ \subset \bZ^3$, with 
\[
J \ll_{d,\eps} 
X^{\eps + B_n(d)},
\qquad
|\cZ| \ll_{d,\eps} 
X^{\eps + 2B_n(d)},
\]
such that the following hold:
\begin{enumerate}[(i)]
\item Each $g_j$ is coprime to $g$ and has degree $O_{d, \eps}(1)$.
\item If $(b,c,y) \in \cY(\bZ) \setminus \cZ$ and
\begin{equation} \label{bcy}
|b| \le CX^{1/2}, \qquad
|c| \le CX^{n/(2n-2)}, \qquad
|y| \le CX^C
\end{equation}
then $g_j(b,c,y) = 0$ for some $j$.
\end{enumerate}
As $d \ge 4$, we have $B_n(d) < N_n$, and the total contribution from $(b,c,y) \in \cZ$ is at most a constant times
\[
X^{\sig_n - 1 - \: \frac1{2n-2} + 2B_n(d) + \eps} < X^E.
\]

Next, given $j$, we count solutions to
\begin{equation} \label{curves}
g(b,c,y) = g_j(b,c,y) = 0.
\end{equation}
In view of \eqref{ybound}, it suffices to count solutions in the ranges \eqref{bcy}. If $\deg_y(g_j) = 0$ then let $F(b,c) = g_j(b,c)$, and otherwise let $F(b,c)$ be the resultant of $g$ and $g_j$ in the variable $y$. Then $F(b,c) = 0$ whenever we have \eqref{curves}, and so 
$\cF(b,c) = 0$ for some irreducible divisor 
\[
\cF(b,c) \in \bQ[b,c]
\]
of $F$. If $\cF$ is non-linear, then Theorem \ref{LopsidedBP} yields
\begin{align*}
&\# \{ (b,c) \in \bZ^2: |b| \le CX^{1/2}, \:
|c| \le CX^{n/(2n-2)}, \: \cF(b,c) = 0 \} \\
&\ll_{d,\eps} X^{N_n + \eps},
\end{align*}
and the contribution to $F_n(X;G)$ from this case is $O_{n,\eps} (X^{E + \eps})$.

Now suppose $\cF$ is linear. If the coefficient of $b$ in $\cF(b,c)$ is non-zero, then $b = c_1 c + c_2$ for some $c_1, c_2 \in \bQ$. As $G < A_n$, we have
\[
P(c,z) := z^2 - \Del(0,a_2,\ldots,a_{n-3}, a, c_1c+c_2,c) = 0
\]
for some $z \in \bN$, where $\Del(a_1,\ldots,a_n)$ denotes the discriminant of 
\[
x^n + a_1 x^{n-1} + \cdots + a_n.
\]
The polynomial $P$ is irreducible, by Lemma \ref{even1}. Note that $\Del$ has total degree $2n-2$, so if 
$|a_1|,\ldots,|a_n| \le X$ then 
\[
|\Del(a_1,\ldots,a_n)| \le CX^{2n-2}.
\]
In particular, if $P(c,z) = 0$ and 
$a_2,\ldots,a_{n-3},a,
c_1c+c_2, c \in [-X,X]$ then $|z| \le CX^{n-1}$. Theorem \ref{LopsidedBP} yields
\begin{align*}
&\# \{ (c,z) \in \bZ^2: 
|c| \le CX^{n/(2n-2)}, \:
|z| \le CX^{n-1}, \:
P(c,z) = 0 \} \\
&\ll_{n,\eps} X^{N_n + \eps}.
\end{align*}
Thus, this case contributes at most $O_{n,\eps} (X^{E + \eps})$ to $F_n(X;G)$.

Otherwise $\cF(b,c) = \lam(c-\mu)$, for some $\mu \in \bQ$ and some $\lam \in \bQ \setminus \{0\}$, and
\[
Q(b,z) := 
z^2 - \Del(0, a_2 ,\ldots,a_{n-3}, a, b,\mu) 
= 0
\]
for some $z \in \bN$. The polynomial $Q$ is irreducible, by Lemma \ref{even2}. Theorem \ref{LopsidedBP} yields
\[
\# \{ (b,z) \in \bZ^2: |b| \le CX^{1/2}, 
\: |z| \le CX^{n-1},
\: Q(b,z) = 0 \} \ll X^{\eps+1/4}.
\]
As $1/4 < N_n$, the contribution to $F_n(X;G)$ from this case is $O_n(X^E)$. 

\subsection{Inseparable case}

Let us now instead choose $a \in \bZ$ such that $g(b,c,y) = \Phi_{f_3}(a,b,c,y)$ is inseparable in $y$. We do not use this information specifically, but rather that there are at most $O(1)$ many such choices.

We have two variables left to specialise, namely $b$ and $c$. We will use Galois theory over function fields to argue that specialising $b$ generically gives Galois group $S_n$. We will then use a Galois resolvent and a lopsided Bombieri--Pila estimate to show that very few specialisations
$c$ lead to Galois group $G \ne S_n$.

To carry out this strategy, we now specialise $b$ and write 
\[
f_1(x;c) = f_2(x;b,c).
\]
By Lemma \ref{OneParameter}, there are at most $O(1)$ integers $b$ such that 
\[
\Gal(f_1, \bQ(c)) \ne S_n.
\]
As
\[
\sig_n - 1 - \frac1{2n-2} - \frac{n-2}{2n-2}
+ \frac{n}{2n-2} = 
\sig_n - 1 + \frac1{2n-2} < E,
\]
this case contributes at most $O_n(X^E)$ to $F_n(X;G)$.

For each of the other $O(X^{1/2})$ possible choices of $b \in \bZ$, we have $\Gal(f_1, \bQ(c)) = S_n$. Lemma \ref{GeneralResolvent} furnishes $\Phi_{f_1}(c,y) 
\in \bZ[c,y]$, monic of degree $d$ in $y$ and irreducible, such that if $c \in \bZ \cap [-X,X]$ and $f_1(x;c)$ has Galois group $G$ over $\bQ$ then $\Phi_{f_1}(c,Y)$ has an integer root $y \ll_n X^{O_n(1)}$. Lemma \ref{LopsidedBP} yields
\begin{align*}
&\# \{ (c,y) \in \bZ^2: |c| \le CX^{n/(2n-2)}, \:
|y| \le CX^C, \: \Phi_{f_1}(c,y) = 0 \} \\
&\ll_{n,\eps} 
X^{\eps + n / (d(2n-2))}.
\end{align*}
Since
\[
\sig_n - 1 - \frac{1}{2n-2} - \frac{n-2}{2n-2}
+ \frac12 + \frac{n}{d(2n-2)} =
\sig_n - 1 + \frac{n}{d(2n-2)} < E,
\]
this final case also contributes at most $O_n(X^E)$ to $F_n(X;G)$.

Therefore 
\[
F_n(X;G) \ll_{n,\eps} X^{E+\eps},
\]
completing the proof of Theorem \ref{MainThm}.

\providecommand{\bysame}{\leavevmode\hbox to3em{\hrulefill}\thinspace}

\end{document}